\documentclass[11pt,reqno]{amsart}
\usepackage[utf8]{inputenc}
\usepackage{amscd,amssymb,amsmath,amsthm}
\usepackage[arrow,matrix]{xy}
\usepackage{graphicx}
\usepackage{epstopdf}
\usepackage{color}
\usepackage{multicol}
\usepackage{tikz}
\newtheorem{thm}{Theorem}
\newtheorem{defn}{Definition}
\newtheorem{lemma}{Lemma}
\newtheorem{pro}{Proposition}

\newtheorem{ex}{Example}

\makeatletter
\@namedef{subjclassname@2020}{%
  \textup{2020} Mathematics Subject Classification}
\makeatother

\numberwithin{equation}{section} 

\begin{document}
\title [Global Dynamics of a Discrete Mosquito Model with Allee Effect]
{Global Dynamics of a Discrete Mosquito Model with Allee Effect}

\author{Z.S. Boxonov}

\address{Z.Boxonov,
V.I.Romanovskiy Institute of Mathematics, 9, University str. 100174, Tashkent, Uzbekistan.}
 \email {z.boxonov@mathinst.uz}

\begin{abstract} In this paper, we investigate the dynamical behavior of a two-dimensional discrete mosquito model incorporating an Allee effect. We analyze the system's trajectory comprehensively, examining both local and global dynamics. Our analysis begins by establishing the invariance of the positive trajectory within the first quadrant. We then investigate the existence of fixed points within this region. Following this, we perform a stability analysis to determine both local and global asymptotic stability of the identified fixed points. Finally, we validate the effectiveness and applicability of our theoretical results through numerical examples.
\end{abstract}


 \maketitle

\section{\textbf{Introduction}}

Analyzing the global dynamics and stability of complex systems remains a significant challenge in discrete dynamical systems. This has spurred recent interest in the global stability of discrete-time systems, a key area within dynamic system theory. In this context, we investigate the global dynamics of a discrete mosquito population model.

Mosquitoes transmit a variety of diseases, including malaria and several viruses, making their population control a crucial public health practice, especially in tropical regions.

Mathematical models play a crucial role in understanding mosquito population dynamics. These models can be continuous-time or discrete-time, as exemplified by the works in \cite{BR1}-\cite{BR4}, \cite{J}-\cite{LP}, \cite{Rpd}, and \cite{RV}.

This work focuses on a discrete-time dynamical system for a simplified mosquito population model presented in \cite{J.Li}. This model incorporates Allee effects, which are situations where population growth is limited at low densities. Mosquitoes undergo complete metamorphosis, with four distinct development stages: egg, larva, pupa, and adult. To simplify the model, we combine the three aquatic stages (egg, larva, pupa) into a single larvae class denoted by $x$. The adult mosquito population is denoted by $y$ \cite{J.Li}. Moreover, we assume that density dependence exists only in the larvae stage.

Denote the birth rate, i.e., the oviposition rate of adults by $\beta(t)$; the rate of emergence from larvae to adults by a function of the larvae with the form of $\alpha(1-k(x))$, where $\alpha>0$ is the maximum emergence rate, $0\leq k(x)\leq 1$, with $k(0)=0, k'(x)>0$, and $\lim\limits_{x\rightarrow
\infty}k(x)=1$, is the functional response due to
the intraspecific competition \cite{J}.
Moreover, we assume the death rate of larvae be a linear function, denoted by $d_{0}+d_{1}x$, and
the death rate of adults be constant, denoted by $\mu$. Then, in the absence of sterile mosquitoes, we get the following
system of equations:
\begin{equation}\label{sys}
\left\{%
\begin{array}{ll}
    \frac{dx}{dt}=\beta(t) y-\alpha(1-k(x))x-(d_{0}+d_{1}x)x,\\[3mm]
    \frac{dy}{dt}=\alpha(1-k(x))x-\mu y. \\
\end{array}%
\right.\end{equation}
    We further assume a functional response for $k(x)$, as in
    \cite{J}, in the form $$k(x)=\frac{x}{1+x}.$$

The dynamical system  (\ref{sys}) was previously studied in \cite{J} for the case where the birth rate $\beta$ is constant (i.e., $\beta(t)=\beta$ ), in which mosquito adults have no difficulty finding their mates such that no Allee effects are concerned. The discrete-time version of this model was then investigated in \cite{BR1}-\cite{BR4}.

A component Allee effect describes a decrease in any fitness component (e.g., survival, reproduction) as population size or density declines.  A prime example is the decrease in a female's probability of mating with fewer males around. This specific phenomenon is commonly referred to as a "mate-finding Allee effect" (\cite{XF}).

When adult mosquitoes face difficulty finding mates, the model incorporates Allee effects. This is reflected in the adult birth rate, which is given by:
$$\beta(t)=\frac{\beta y(t)}{\gamma+y(t)},$$
where $\beta$ is birth rate, $\gamma$ is the Allee effect constant. Then stage-structured wild mosquito population model is given by:
\begin{equation}\label{system}
\left\{%
\begin{array}{ll}
    \frac{dx}{dt}=\frac{\beta y^2}{\gamma+y}-\frac{\alpha x}{1+x}-(d_{0}+d_{1}x)x,\\[3mm]
    \frac{dy}{dt}=\frac{\alpha x}{1+x}-\mu y. \\
\end{array}%
\right.\end{equation}
This continuous system has been discussed in \cite{J.Li}, but its discrete version has not been investigated as of yet.
Discretization of the continuous-time model offers several advantages for studying mosquito dynamics. It allows simulations to be conducted on discrete time scales, which are more relevant to biological processes. This facilitates a closer alignment with actual mosquito population data and enables the effective integration of time-sensitive ecological components.
Continuous systems can be approached by several discrete methods, including Euler forward/backward and semi-discretization. We utilize the Euler forward difference scheme here.
In this paper, we study the discrete time dynamical systems associated with the system (\ref{system}) by employing Euler's forward discretization method.
Define the operator
$W:{\mathbb{R}^2}\rightarrow {\mathbb{R}^2}$ by
\begin{equation}\label{systema}
\left\{%
\begin{array}{ll}
    x'=\frac{\beta y^2}{\gamma+y}-\frac{\alpha x}{1+x}-(d_{0}+d_{1}x)x+x,\\[3mm]
    y'=\frac{\alpha x}{1+x}-\mu y+y, \\
\end{array}%
\right.\end{equation}
where $\alpha >0, \beta >0, \gamma>0, \mu >0,  \ d_{0}\geq0,\ d_{1}\geq0.$ In the system (\ref{systema}), $x', y'$ means the next state relative to the initial state $x,y$ respectively.

If for each possible state $\textbf{z}=(x,y)\in\mathbb{R}^{2}$ describing the current generation, the state $\textbf{z}'=(x',y')\in\mathbb{R}^{2}$ is uniquely defined as $\textbf{z}'=W(\textbf{z})$ then the map $W:\mathbb{R}^{2}\rightarrow \mathbb{R}^{2}$ called the evolution operator \cite{L}, \cite{Rpd}.

The main problem for a given operator $W$ and arbitrarily initial point $\textbf{z}^{(0)}=(x^{(0)},y^{(0)})\in\mathbb{R}^{2}$  is to describe the limit points of the trajectory $\{\textbf{z}^{(m)}\}_{m=0}^{\infty}$, where $\textbf{z}^{(m)}=W^m(\textbf{z}^{(0)}).$

In \cite{RB} the dynamics of the operator defined by (\ref{systema}) is studied under the condition $d_0=d_1=0$. In this paper we consider the operator $W$ (defined by (\ref{systema})) for the case $d_0>0$.

\section{\bf Preliminaries and Known Results}

Consider systems of difference equations of the form
\begin{equation}\label{xy}
\left\{%
\begin{array}{ll}
    x_{n+1}=f_1(x_n, y_n),\\[2mm]
    y_{n+1}=f_2(x_n, y_n),
\end{array} n=0,1,2,... %
\right.\end{equation}
where $f_1$ and $f_2$ are given functions and the initial condition $(x_0, y_0)$ comes from some considered set in the intersection of the domains of $f_1$ and $f_2$. An operator $W$ that corresponds to the system (\ref{xy}) is defined as $W(x, y)=\left(f_1(x, y), f_2(x, y)\right)$.

Let $\mathfrak{R}$ be a subset of $\mathbb{R}^2$ with a nonempty interior, and let $W:\mathfrak{R}\rightarrow\mathfrak{R}$ be a continuous operator. When the function $f_1(x, y)$ is increasing in $x$ and decreasing in $y$ and the function $f_2(x, y)$ is decreasing in $x$ and increasing in $y$, the system (\ref{xy}) is called competitive. When the function $f_1(x, y)$ is increasing in $x$ and increasing in $y$ and the function $f_2(x, y)$ is increasing in $x$ and increasing in $y$, the system (\ref{xy}) is called cooperative. Competitive and cooperative maps, which are called monotone maps, are defined similarly. Strongly cooperative systems of difference equations or maps are those for which the functions $f_1$ and $f_2$ are coordinate-wise strictly monotone \cite{BBK}.

Let $\preceq$ be a partial order on $\mathbb{R}^n$ with nonnegative cone $P$. For $z_1,z_2\in\mathbb{R}^n$ the order interval $\mathopen{[\![} z_1, z_2 \mathclose{]\!]}$ is the set of all $z$ such that $z_1\preceq z\preceq z_2$. We say $z_1\prec z_2$ if $z_1\preceq z_2$ and $z_1\neq z_2$, and $z_1\ll z_2$ if $z_2-z_1\in int(P)$. An operator $W$ on a subset of $\mathbb{R}^n$ is order-preserving if $W(z_1)\preceq W(z_2)$ whenever $z_1\prec z_2$, strictly order-preserving $W(z_1)\prec W(z_2)$ whenever $z_1\prec z_2$, and strongly order-preserving if $W(z_1)\ll W(z_2)$ whenever $z_1\prec z_2$.

A point $(x^*, y^*)\in\mathbb{R}^{2}$ is called a fixed point of $W$ if $W\left((x^*, y^*)\right)=(x^*, y^*)$.
The basin of attraction of a fixed point $(x^*, y^*)$ of an operator $W$, is defined as the set of all initial points $\left(x^{(0)}, y^{(0)}\right)$ in the domain of $W$ such that the sequence of iterates $W^n\left(x^{(0)}, y^{(0)}\right)$ converges to $(x^*, y^*)$.

The following result is a direct consequence of the Trichotomy Theorem of Dancer and Hess, (see \cite{DH}, \cite{H}, \cite{KM}) proves useful in determining the basins of attraction for the fixed points.
\begin{pro}\label{basin}
If the non-negative cone of a partial ordering $\preceq$ is a generalized quadrant in $\mathbb{R}^n$, and if $W$ has no fixed points in $\mathopen{[\![} z_1, z_2 \mathclose{]\!]}$ other than $z_1$ and $z_2$, then the interior of $\mathopen{[\![} z_1, z_2 \mathclose{]\!]}$ is either a subset of the basin of attraction of $z_1$ or a subset of the basin of attraction of $z_2$.
\end{pro}

Consider the North-East ordering (NE) on $\mathbb{R}^2$ for which the positive cone is the first quadrant, i.e., this partial ordering is defined by $(x_1, y_1)\preceq_{NE}(x_2, y_2)$ if $x_1\leq x_2$ and $y_1\leq y_2$. The South-East (SE) ordering defined as $(x_1, y_1)\preceq_{SE}(x_2, y_2)$ if $x_1\leq x_2$ and $y_1\geq y_2$.

A map $W$ on a nonempty set $\mathfrak{R}\subset\mathbb{R}^2$ which is monotone with respect to the $NE$ ordering is called cooperative and a map monotone with respect to the $SE$ ordering is called competitive.

We define the non-negative quadrant of the plane as $\mathbb{R}_{+}^{2}=\{(x,y): x,y\in\mathbb{R}, x\geq0, y\geq0\}$. The operator $W$ is well-defined on $\mathbb{R}^2$ except for the lines $x=-1$ and $y=-\gamma$. However, to model a population dynamics system using a continuous operator, we require both population values (represented by $x$ and $y$) to be non-negative. Therefore, we restrict our analysis to $\mathbb{R}^2_{+}$ and choose the parameters of the operator $W$ such that it maps this set onto itself.

\begin{lemma}\label{0-0} If
\begin{equation}\label{parametr}
\alpha>0,\ \beta>0,\ \gamma>0,\ 0<\mu\leq1,\ d_{0}>0,\ \alpha+d_{0}\leq1,\ d_{1}=0
\end{equation}
then the operator (\ref{systema}) maps the set $\mathbb{R}_{+}^{2}$ to itself.
\end{lemma}        
\begin{proof} This can be proved by direct computation.
\end{proof}
Lemma \ref{0-0} provides sufficient conditions for non-negative values of $x$ and $y$.
In this case the system (\ref{systema}) becomes
\begin{equation}\label{systemacase}
W_{0}:\left\{%
\begin{array}{ll}
    x'=\frac{\beta y^2}{\gamma+y}+\left(1-d_0-\frac{\alpha}{1+x}\right)x,\\[3mm]
    y'=\frac{\alpha x}{1+x}+(1-\mu)y. \\
\end{array}%
\right.\end{equation}

\begin{defn}\label{d1}
(see \cite{D}) A fixed point $z$ of an operator $W_0$ is called
\texttt{hyperbolic} if its Jacobian $J$ at $z$ has no eigenvalues on the
unit circle.
\end{defn}

\begin{defn}\label{d2}
(see \cite{D}) A hyperbolic fixed point $z$ is called:
\begin{itemize}
\item[1)] \texttt{attracting} if all the eigenvalues of the Jacobian matrix $J(z)$
are less than 1 in absolute value;

\item[2)] \texttt{repelling} if all the eigenvalues of the Jacobian matrix $J(z)$
are greater than 1 in absolute value;

\item[3)] a \texttt{saddle} otherwise.
\end{itemize}
\end{defn}
To investigate the local stability analysis of the fixed point of system (\ref{systemacase}), we present the following lemma \cite{ChX}.
\begin{lemma}\label{F(l)}
Let $F(\lambda)=\lambda^2+B\lambda+C,$ where $B$ and $C$ are two real constants. Suppose $\lambda_1$ and $\lambda_2$ are two roots of $F(\lambda)=0$. Then the following statements hold.
\begin{enumerate}
  \item[(i)] If $F(1)>0$ then
  \item[(i.1)] $|\lambda_1|<1$ and $|\lambda_2|<1$ if and only if $F(-1)>0$ and $C<1;$
  \item[(i.2)] $\lambda_1=-1$ and $\lambda_2\neq-1$ if and only if $F(-1)=0$ and $B\neq2;$
  \item[(i.3)] $|\lambda_1|<1$ and $|\lambda_2|>1$ if and only if $F(-1)<0;$
  \item[(i.4)] $|\lambda_1|>1$ and $|\lambda_2|>1$ if and only if $F(-1)>0$ and $C>1;$
  \item[(i.5)] $\lambda_1$ and $\lambda_2$ are a pair of conjugate complex roots and $|\lambda_1|=|\lambda_2|=1$ if only if $-2<B<2$ and $C=1;$
  \item[(i.6)] $\lambda_1=\lambda_2=-1$ if only if $F(-1)=0$ and $B=2.$
  \item[(ii)] If $F(1)=0,$ namely, 1 is one root of $F(\lambda)=0,$ then the other root $\lambda$ satisfies $|\lambda|=(<, >)1$  if and only if $|C|=(<, >)1.$
  \item[(iii)] If $F(1)<0$ then $F(\lambda)=0$ has one root lying in $(1;\infty).$ Moreover,
  \item[(iii.1)] the other root $\lambda$ satisfies $\lambda<(=)-1$ if and only if $F(-1)<(=)0;$
  \item[(iii.2)] the other root $\lambda$ satisfies $-1<\lambda<1$ if and only if $F(-1)>0.$
\end{enumerate}
\end{lemma}

\section{\bf Analyzing the Existence and Local Stability of Fixed Points}

We begin by investigating the existence of fixed points.

The fixed points of the system (\ref{systemacase}) satisfy the following system of equations:
\begin{equation}\label{Fsystema}
\begin{array}{ll}
    x=\frac{\beta y^2}{\gamma+y}+\left(1-d_0-\frac{\alpha}{1+x}\right)x,\\[3mm]
    y=\frac{\alpha x}{1+x}+(1-\mu)y. \\
\end{array}%
\end{equation}
It can be seen that the point $(0,0)$ is a solution to system (\ref{Fsystema}).
By solving the second equation of system (\ref{Fsystema}) for $y$ and substituting it into the first equation, we obtain the following equation:
\begin{equation}\label{x^3}
d_0\mu(\alpha+\gamma\mu)x^2+\left(\mu\left(d_0\gamma\mu+(\alpha+d_0)(\alpha+\gamma\mu)\right)-\alpha^2\beta\right)x+\gamma\mu^2(\alpha+d_0)=0.
\end{equation}
We denote
\begin{equation}
\begin{array}{ll}
D=\left(\mu(d_0\gamma\mu+(\alpha+d_0)(\alpha+\gamma\mu))-\alpha^2\beta\right)^2-4d_0\gamma\mu^3(\alpha+d_0)(\alpha+\gamma\mu),\\[3mm]
\nu=\frac{\mu}{\alpha^2}\left(\sqrt{d_0\gamma\mu}+\sqrt{(\alpha+d_0)(\alpha+\gamma\mu)}\right)^2.
\end{array}%
\end{equation}
Conditions in (\ref{parametr}) ensure $d_0\mu(\alpha+\gamma\mu)>0$ and $\gamma\mu^2(\alpha+d_0)>0$. For positive roots in equation (\ref{x^3}), we additionally require $\mu\left(d_0\gamma\mu+(\alpha+d_0)(\alpha+\gamma\mu)\right)-\alpha^2\beta<0$ and $D>0$.
From here we get $\beta>\nu.$ The number of solutions to equation (\ref{x^3}) depends on the relationship between $\beta$ and $\nu$. If $\beta=\nu$, equation (\ref{x^3}) has one solution. If $\beta<\nu$, equation (\ref{x^3}) has no solution.

The following proposition holds for fixed points of the operator $W_0$:
\begin{pro}\label{fixed} The fixed points for (\ref{systemacase}) are as follows:
\begin{itemize}
  \item If $\beta<\nu$ then the operator (\ref{systemacase}) has a unique fixed point $z_1=(0,0).$
  \item If $\beta=\nu$ then mapping (\ref{systemacase}) has two
fixed points $$z_{1,1}=(0,0), \ \ z_{1,2}=\left(x^*_1, y^*_1\right).$$
\item If $\beta>\nu$ then mapping (\ref{systemacase}) has three
fixed points $$z_{2,1}=(0,0), \ \ z_{2,2}=\left(x^*_2, y^*_2\right), \ \ z_{2,3}=\left(x^*_3, y^*_3\right),$$
\end{itemize}
where
\begin{equation}\label{fix123}
\begin{array}{llll}
x^*_1=\sqrt{\frac{\gamma\mu(\alpha+d_0)}{d_0(\alpha+\gamma\mu)}}, \\[3mm]
x^*_2=\frac{\alpha^2\beta-\mu(d_0\gamma\mu+(\alpha+d_0)(\alpha+\gamma\mu))-\sqrt{D}}{2d_0\mu(\alpha+\gamma\mu)},\\[3mm]
x^*_3=\frac{\alpha^2\beta-\mu(d_0\gamma\mu+(\alpha+d_0)(\alpha+\gamma\mu))+\sqrt{D}}{2d_0\mu(\alpha+\gamma\mu)},\\[3mm]
y^*_i=\frac{\alpha x^*_i}{\mu(1+x^*_i)}, \ i=1,2,3.
\end{array}
\end{equation}
\end{pro}

Now we shall examine the type of the fixed point.

To determine the type of a fixed point $z=(x^*,y^*)$ of the operator defined by (\ref{systemacase}), we can analyze the Jacobian matrix:

$$J_{W_{0}}=\left(%
\begin{array}{cc}
  1-d_0-\frac{\alpha}{(1+x)^2} & \frac{\beta y(2\gamma+y)}{(\gamma+y)^2} \\
  \frac{\alpha}{(1+x)^2} & 1-\mu \\
\end{array}%
\right).$$

The characteristic equation is
\begin{equation}\label{xar.eq}
F(\lambda)=\lambda^2+B\lambda+C,
\end{equation}
where $$B=\mu+d_0+\frac{\alpha}{(1+x^*)^2}-2,$$ $$C=1-\mu-d_0-\frac{\alpha}{(1+x^*)^2}+\mu d_0+\frac{\alpha}{(1+x^*)^2}\left(\mu-\frac{\beta y^*(2\gamma+y^*)}{(\gamma+y^*)^2}\right).$$

From (\ref{xar.eq}) we get
\[F(1)=1+B+C=\mu\left(d_0+\frac{\alpha}{(1+x^*)^2}\right)-\frac{\alpha}{(1+x^*)^2}\cdot\frac{\beta y^*}{\gamma+y^*}\left(1+\frac{\gamma}{\gamma+y^*}\right),\]
\[F(-1)=1-B+C=4-2\left(\mu+d_0+\frac{\alpha}{(1+x^*)^2}\right)+\mu\left(d_0+\frac{\alpha}{(1+x^*)^2}\right)-\] \[\frac{\alpha}{(1+x^*)^2}\cdot\frac{\beta y^*}{\gamma+y^*}\left(1+\frac{\gamma}{\gamma+y^*}\right).\]

$\textbf{(1)}$ Let $\beta<\nu$, i.e., $x^*=0, y^*=0.$ Then \[F(1)=\mu(\alpha+d_0)>0,\ C=(1-d_0-\alpha)(1-\mu)<1,\]
\[F(-1)=2(2-\mu-d_0-\alpha)+\mu(d_0+\alpha)>0.\]
According to item (i.1) of Lemma \ref{F(l)}, $|\lambda_{1,2}|<1.$ Therefore, the fixed point $(0,0) $is attracting.

$\textbf{(2)}$ Let $\beta=\nu$, i.e., $x^*=x_1^*=\sqrt{\frac{\gamma\mu(\alpha+d_0)}{d_0(\alpha+\gamma\mu)}}.$ Then
\[F(1)=\mu\left(d_0+\frac{\alpha}{(1+x^*)^2}\right)-\frac{\alpha}{(1+x^*)^2}\cdot\frac{\beta y^*}{\gamma+y^*}\left(1+\frac{\gamma}{\gamma+y^*}\right)\]
\[=\mu\left(d_0+\frac{\alpha}{(1+x^*)^2}\right)-\frac{\mu y^*}{(1+x^*)x^*}\cdot\frac{\beta y^*}{\gamma+y^*}\left(1+\frac{\gamma}{\gamma+y^*}\right)\]
\[=\frac{\mu}{(1+x^*)x^*}\left(d_0x^*(1+x^*)+\frac{\alpha x^*}{1+x^*}\right)-\frac{\mu y^*}{(1+x^*)x^*}\cdot\frac{\beta y^*}{\gamma+y^*}\left(1+\frac{\gamma}{\gamma+y^*}\right)\]
\[=\frac{\mu}{(1+x^*)x^*}\left(d_0x{^*}^2+\frac{\beta y{^*}^2}{\gamma+y^*}\right)-\frac{\mu y^*}{(1+x^*)x^*}\cdot\frac{\beta y^*}{\gamma+y^*}\left(1+\frac{\gamma}{\gamma+y^*}\right)\]
\[=\frac{\mu}{(1+x^*)x^*}\left(d_0x{^*}^2-\frac{\beta\gamma y{^*}^2}{(\gamma+y^*)^2}\right)=\frac{\mu}{(1+x^*)x^*}\left(d_0x{^*}^2-\frac{\beta\gamma\alpha^2 x{^*}^2}{(\gamma\mu+(\alpha+\gamma\mu)x^*)^2}\right)\]
\[=\frac{\mu x{^*}^2}{(1+x^*)x^*}\left(d_0-\frac{\beta\gamma\alpha^2}{(\gamma\mu+(\alpha+\gamma\mu)x^*)^2}\right)\]
\[=\frac{\mu x^*}{1+x^*}\left(d_0-\frac{\gamma \alpha^2\mu\left(\frac{\sqrt{d_0\gamma\mu}+\sqrt{(\alpha+d_0)(\alpha+\gamma\mu)}}{\alpha}\right)^2}{\left(\gamma\mu+(\alpha+\gamma\mu)\sqrt{\frac{\gamma\mu(\alpha+d_0)}{d_0(\alpha+\gamma\mu)}}\right)^2}
\right)=\frac{\mu x^*}{1+x^*}\left(d_0-d_0\right)=0,\]
\[-1<C=1-\mu-d_0-\frac{\alpha}{(1+x^*)^2}+F(1)=1-\mu-d_0-\frac{\alpha}{(1+x^*)^2}<1.\]
According to item (ii) of Lemma \ref{F(l)}, $\lambda_{1}=1$, $|\lambda_{2}|<1.$ So, the fixed point $(x_1^*,y_1^*) $ is non-hyperbolic.

$\textbf{(3)}$ Let $\beta>\nu$, i.e., $x^*=x_{2,3}^*$. Then
\[F(1)=\frac{\mu x{^*}^2}{(1+x^*)x^*}\left(d_0-\frac{\beta\gamma\alpha^2}{(\gamma\mu+(\alpha+\gamma\mu)x^*)^2}\right)\]
\[=\frac{\mu x^*}{1+x^*}\left(d_0-\frac{\beta\gamma \alpha^2}{\left(\gamma\mu+(\alpha+\gamma\mu)\cdot\frac{\alpha^2\beta-\mu(d_0\gamma\mu+(\alpha+d_0)(\alpha+\gamma\mu))\pm\sqrt{D}}{2d_0\mu(\alpha+\gamma\mu)}\right)^2}\right)\]
\[=\frac{\mu d_0x^*}{1+x^*}\cdot\frac{D\pm\sqrt{D^2+4\alpha^2\beta\gamma\mu^2d_0D}}{D+2\alpha^2\beta\gamma\mu^2d_0\pm\sqrt{D^2+4\alpha^2\beta\gamma\mu^2d_0}}.\]

If $x^*=x_{3}^*$, then $F(1)>0$, $F(-1)>0$, $C<1.$ According to item (i1) of Lemma \ref{F(l)}, $|\lambda_{1,2}|<1.$  Therefore, the fixed point $(x_3^*,y_3^*) $ is attracting.
If $x^*=x_{2}^*$, then $F(1)<0$, $F(-1)<0.$ According to item (iii1) of Lemma \ref{F(l)}, $\lambda_{1}>1$, $\lambda_{2}<-1.$  So, the fixed point $(x_2^*,y_2^*) $ is repelling.

Thus for the type of fixed points the following theorem holds.

\begin{thm}\label{type} The type of the fixed points for (\ref{systemacase}) are as follows:
\begin{itemize}
  \item[i)] if $\beta<\nu$,  then the operator (\ref{systemacase}) has unique fixed point $(0,0)$ which is attracting.
  \item[ii)]if $\beta=\nu$, then the operator (\ref{systemacase}) has two fixed points $(0,0)$, $(x_1^*,y_1^*)$, and the point $(0,0)$ is attracting, the point $(x_1^*,y_1^*)$ is non-hyperbolic (but $(x_1^*,y_1^*)$ is semi-attracting \footnote{meaning that the second eigenvalue is less than 1 in absolute value.}).
  \item[iii)]if $\beta>\nu$, then the operator (\ref{systemacase}) has three fixed points $(0,0)$, $(x_{2}^*,y_{2}^*)$, $(x_{3}^*,y_{3}^*)$, and the points $(0,0)$ and $(x_{3}^*,y_{3}^*)$ are attracting, the point $(x_2^*,y_2^*)$ is repelling.
  \end{itemize}
where $(x_1^*,y_1^*)$, $(x_2^*,y_2^*)$ and $(x_3^*,y_3^*)$ are defined by (\ref{fix123}).
\end{thm}

%
%

\section{\textbf{Global Dynamics of the Operator $W_0$}}

In this section for any initial point $\left(x^{(0)}, y^{(0)}\right)\in \mathbb{R}^{2}_{+}$ we investigate behavior of the trajectories $\left(x^{(n)},y^{(n)}\right)=W_0^n\left(x^{(0)}, y^{(0)}\right), n\geq1.$

Let
\begin{equation}\label{Omega12}
\Omega=[0,\omega_1]\times[0,\omega_2]
\end{equation}
with
\[\omega_1=\frac{\alpha^2\beta}{\mu d_{0}(\alpha+\gamma\mu)}, \ \omega_2=\frac{\alpha}{\mu}.\]

\begin{defn}(see \cite{NW})
The set $A$ is called the absorbing set for the operator $W_0$ if it satisfies the following properties:
\begin{itemize}
  \item [1)] Invariance: the set $A$ is invariant with respect to $W_{0}$, i.e., $W_{0}(A)\subset A$;
  \item [2)] Eventual absorption: for any initial point $\left(x^{(0)}, y^{(0)}\right)\in \mathbb{R}^{2}_{+}$, there exists $k_0\in \mathbb{N}$ such that $W_0^k\left(x^{(0)}, y^{(0)}\right)\in A$ for all $k\geq k_0.$
\end{itemize}
\end{defn}
\begin{lemma}\label{ab.set}(Absorbing set).The set $\Omega$ is an absorbing set for the operator $W_0$ given by (\ref{systemacase}). 
\end{lemma}
\begin{proof}

We have
\begin{equation}\label{x^ny^n}
\begin{array}{ll}
x^{(n)}=\frac{\beta\left(y^{(n-1)}\right)^2}{\gamma+y^{(n-1)}}+\left(1-d_{0}-\frac{\alpha}{1+x^{(n-1)}}\right)x^{(n-1)},\\[2mm]
y^{(n)}=\frac{\alpha x^{(n-1)}}{1+x^{(n-1)}}+\left(1-\mu\right)y^{(n-1)}.
\end{array}
\end{equation}
Assume that all values of $y^{(n)}$ are greater than $\omega_2$. Then
$$y^{(n+1)}-y^{(n)}=\frac{\alpha x^{(n)}}{1+x^{(n)}}-\mu y^{(n)}<\alpha-\mu y^{(n)}=\mu\left(\omega_2-y^{(n)}\right)<0.$$ So $y^{(n)}$ is a decreasing sequence. Since  $y^{(n)}$ is decreasing and bounded from below we have:
\begin{equation}\label{y[n]>a}
\lim_{n\to \infty}y^{(n)}\geq\omega_2.
\end{equation}
We estimate $y^{(n)}$ by the following:
 $$y^{(n)}=\frac{\alpha x^{(n-1)}}{1+x^{(n-1)}}+(1-\mu)y^{(n-1)}<\alpha+(1-\mu)y^{(n-1)}<\alpha+(1-\mu)\left(\alpha+(1-\mu)y^{(n-2)}\right)$$
 $$<\alpha+\alpha(1-\mu)+(1-\mu)^2\left(\alpha+(1-\mu)y^{(n-3)}\right)< ...$$ $$<\alpha+\alpha(1-\mu)+\alpha(1-\mu)^2+...+\alpha(1-\mu)^{n-1}+(1-\mu)^{n}y^{(0)}=\omega_2+(1-\mu)^{n}\left(y^{(0)}-\omega_2\right).$$
Thus $y^{(n)}<\omega_2+(1-\mu)^{n}\left(y^{(0)}-\omega_2\right)$. Consequently
\begin{equation}\label{y[n]<a}
\lim_{n\to \infty}y^{(n)}\leq\omega_2.
\end{equation}
 By (\ref{y[n]>a}) and (\ref{y[n]<a}) we have \begin{equation}\label{y[n]=a}
 \lim\limits_{n\to \infty}y^{(n)}=\omega_2.
\end{equation}
From (\ref{y[n]=a}) and $y^{(n)}=\frac{\alpha x^{(n-1)}}{1+x^{(n-1)}}-\mu y^{(n-1)}+y^{(n-1)}$ it follows
\begin{equation}\label{x[n]=infty}\lim\limits_{n\to \infty}x^{(n)}=+\infty.
\end{equation}

Now we estimate the sequence $x^{(n)}$. The inequality $y^{(n)}\leq y^{(0)}$ holds for all $n$, as the sequence $y^{(n)}$ is decreasing. We denote $y_0^{(0)}=\frac{\beta \left(y^{(0)}\right)^2}{\gamma+y^{(0)}}.$ Then
\[x^{(n)}=\frac{\beta \left(y^{(n-1)}\right)^2}{\gamma+y^{(n-1)}}+\left(1-d_{0}-\frac{\alpha}{1+x^{(n-1)}}\right)x^{(n-1)}<y_0^{(0)}+(1-d_{0})x^{(n-1)}\]
\[<y_0^{(0)}+(1-d_{0})\left(y_0^{(0)}+(1-d_{0})x^{(n-2)}\right)<...\]
\[<y_0^{(0)}+y_0^{(0)}(1-d_{0})+...+y_0^{(0)}(1-d_{0})^{n-1}+(1-d_{0})^{n}x^{(0)}\]
\[=\frac{y_0^{(0)}}{d_{0}}+(1-d_{0})^{n}\left(x^{(0)}-\frac{y_0^{(0)}}{d_{0}}\right).\]
From condition (\ref{parametr}) it follows $d_{0}\in(0, 1).$ Consequently
\begin{equation}\label{x[n]<y0}
\lim_{n\to \infty}x^{(n)}\leq\frac{y_0^{(0)}}{d_{0}}.
\end{equation}
Our initial assumption is invalid based on  (\ref{x[n]=infty}) and (\ref{x[n]<y0}). Therefore, there exists $n_0$ such that $y^{(n_0)}$ is not greater than $\omega_2$.
If $y^{(n_0)}\leq\omega_2$ then $y^{(n_0+1)}<\omega_2$. Indeed,
  \[y^{(n_0+1)}=(1-\mu)y^{(n_0)}+\frac{\alpha x^{(n_0)}}{1+x^{(n_0)}}\leq(1-\mu)\omega_2+\frac{\alpha x^{(n_0)}}{1+x^{(n_0)}}\]
  \[=\omega_2-\alpha\left(1-\frac{x^{(n_0)}}{1+x^{(n_0)}}\right)<\omega_2.\]

Assume that all values of $x^{(n)}$ are greater than $\omega_1$. Then
$$x^{(n+1)}-x^{(n)}=\frac{\beta \left(y^{(n)}\right)^2}{\gamma+y^{(n)}}-\frac{\alpha x^{(n)}}{1+x^{(n)}}-d_{0}x^{(n)}<\frac{\beta \omega_2^2}{\gamma+\omega_2}-d_{0} x^{(n)}=d_{0}\left(\omega_1-x^{(n)}\right)\leq0.$$ So $x^{(n)}$ is a decreasing sequence. Since  $x^{(n)}$ is decreasing and bounded from below we have:
\begin{equation}\label{x[n]>a}
\lim_{n\to \infty}x^{(n)}\geq\omega_1.
\end{equation}
We estimate $x^{(n)}$ by the following:
\[x^{(n)}=\frac{\beta \left(y^{(n-1)}\right)^2}{\gamma+y^{(n-1)}}+\left(1-d_0-\frac{\alpha}{1+x^{(n-1)}}\right)x^{(n-1)}<\frac{\beta \omega_2^2}{\gamma+\omega_2}+(1-d_{0})x^{(n-1)}\]
\[<\frac{\beta \omega_2^2}{\gamma+\omega_2}+(1-d_{0})\left(\frac{\beta \omega_2^2}{\gamma+\omega_2}+(1-d_{0})x^{(n-2)}\right)<...\]
\[<\frac{\beta \omega_2^2}{\gamma+\omega_2}+\frac{\beta \omega_2^2}{\gamma+\omega_2}(1-d_{0})+...+\frac{\beta \omega_2^2}{\gamma+\omega_2}(1-d_{0})^{n-1}+(1-d_{0})^{n}x^{(0)}\]
\[=\omega_1+(1-d_{0})^{n}\left(x^{(0)}-\omega_1\right).\]
From $d_{0}\in(0, 1)$, consequently, the inequality holds.
\begin{equation}\label{x[n]<a}
\lim_{n\to \infty}x^{(n)}\leq\omega_1.
\end{equation}
 By (\ref{x[n]>a}) and (\ref{x[n]<a}) we have
 \begin{equation}\label{x[n]=a}
 \lim\limits_{n\to \infty}x^{(n)}=\omega_1.
\end{equation}
Since $x^{(n)}$ converges to $\omega_1$, it follows from the second equation of (\ref{x^ny^n}) that $$\lim\limits_{n\to \infty}\left(y^{(n)}-(1-\mu)y^{(n-1)}\right)=\frac{\alpha\omega_1}{1+\omega_1}.$$
Since $y^{(n)}$ is bounded, it possesses upper and lower limits, denoted by $\overline{y}$ and $\underline{y}$, respectively. To arrive at a contradiction, suppose $y^{(n)}$ has no limit (i.e.,$\overline{y}\neq\underline{y}$). This would necessitate the existence of two subsequences: $y^{(n_i)}$ converging to $\overline{y}$ and $y^{(n_j)}$ converging to $\underline{y}$.
\begin{flushleft}
$\lim\limits_{i\to \infty}\left(y^{(n_i)}-(1-\mu)y^{(n_i-1)}\right)=\frac{\alpha\omega_1}{1+\omega_1},\ \Rightarrow \ \lim\limits_{i\to \infty}y^{(n_i-1)}=\frac{1}{1-\mu}\left(\overline{y}-\frac{\alpha\omega_1}{1+\omega_1}\right)\leq\overline{y}, \ \Rightarrow \ \overline{y}\leq\frac{\alpha\omega_1}{\mu(1+\omega_1)};$\\[2mm]
$\lim\limits_{i\to \infty}\left(y^{(n_i+1)}-(1-\mu)y^{(n_i)}\right)=\frac{\alpha\omega_1}{1+\omega_1},\ \Rightarrow \ \lim\limits_{i\to \infty}y^{(n_i+1)}=(1-\mu)\overline{y}+\frac{\alpha\omega_1}{1+\omega_1}\leq\overline{y},\ \Rightarrow \ \overline{y}\geq\frac{\alpha\omega_1}{\mu(1+\omega_1)};$\\[2mm]
$\lim\limits_{j\to \infty}\left(y^{(n_j)}-(1-\mu)y^{(n_j-1)}\right)=\frac{\alpha\omega_1}{1+\omega_1},\ \Rightarrow \ \lim\limits_{j\to \infty}y^{(n_j-1)}=\frac{1}{1-\mu}\left(\underline{y}-\frac{\alpha\omega_1}{1+\omega_1}\right)\geq\underline{y}, \Rightarrow  \underline{y}\geq\frac{\alpha\omega_1}{\mu(1+\omega_1)};$\\[2mm]
$\lim\limits_{j\to \infty}\left(y^{(n_j+1)}-(1-\mu)y^{(n_j)}\right)=\frac{\alpha\omega_1}{1+\omega_1},\ \Rightarrow \ \lim\limits_{j\to \infty}y^{(n_j+1)}=(1-\mu)\underline{y}+\frac{\alpha\omega_1}{1+\omega_1}\geq\underline{y},\ \Rightarrow \ \underline{y}\leq\frac{\alpha\omega_1}{\mu(1+\omega1)}.$
\end{flushleft}
Therefore, we conclude that $\overline{y}=\underline{y}=\frac{\alpha\omega_1}{\mu(1+\omega_1)}$ i.e.,
\[\lim\limits_{n\to \infty}y^{(n)}=\frac{\omega_1\omega_2}{1+\omega_1}.\]

Proposition \ref{fixed} shows that $\left(\omega_1, \frac{\omega_1\omega_2}{1+\omega_1}\right)$ is not a fixed point of system (\ref{systemacase}), contradicting our hypothesis.
Therefore, there exists $m_0$ such that $x^{(k_0)}$ is not greater than $\omega_1$.
If $x^{(k_0)}\leq\omega_1$ then $x^{(k_0+1)}<\omega_1$. Indeed,
\[x^{(k_0+1)}=\frac{\beta \left(y^{(k_0)}\right)^2}{\gamma+y^{(k_0)}}+\left(1-d_0-\frac{\alpha}{1+x^{(k_0)}}\right)x^{(k_0)}<\frac{\beta\omega_2^2}{\gamma+\omega_2}+(1-d_{0})x^{(k_0)}\]
\[\leq\frac{\beta\omega_2^2}{\gamma+\omega_2}+(1-d_{0})\omega_1=\omega_1.\]
So for any initial point $\left(x^{(0)}, y^{(0)}\right)\in \mathbb{R}^{2}_{+}$, there exists $k_0\in \mathbb{N}$ such that $W_0^k\left(x^{(0)}, y^{(0)}\right)\in\Omega$ for all $k\geq k_0.$ Furthermore, $W_0$ maps the set $\Omega$ into itself.
\end{proof}

Lemma \ref{ab.set} establishes that $\Omega$ is an absorbing set. Consequently, we can restrict our analysis to the dynamics of the operator $W_0$ within the set $\Omega$.

\begin{pro}\label{f1f2increasing}
Let $[a,b]$ and $[c,d]$ be intervals of real numbers, and let
\[f_1:[a,b]\times[c,d]\rightarrow [a,b], \ \ and \ \ f_2:[a,b]\times[c,d]\rightarrow [c,d]\]
be continuous functions. Consider the system of difference equation
\begin{equation}\label{f1f2}
\begin{array}{ll}
    x_{n+1}=f_1(x_n, y_n),\\[2mm]
    y_{n+1}=f_2(x_n, y_n),
\end{array} \ \ n=0,1,2,... %
\end{equation}
with initial condition $\left(x^{(0)},y^{(0)}\right)\in [a,b]\times[c,d].$ Suppose the following conditions (i) and (ii) are satisfied:
\begin{itemize}
  \item[ \textbf{(i)}] $f_1(x,y)$ and $f_2(x,y)$ are non-decreasing in both $x$ and $y.$
  \item[ \textbf{(ii)}] If $a_1,b_1\in[a,b]$, $c_1,d_1\in[c,d]$ is a solution of the system of equations
  \begin{equation}
\left\{%
\begin{array}{ll}
    a_1=f_1(a_1,c_1), \\[2mm]
    c_1=f_2(a_1,c_1),
\end{array}%
\right.
\left\{%
\begin{array}{ll}
    b_1=f_1(b_1,d_1)\\[2mm]
    d_1=f_2(b_1,d_1)
\end{array}%
\right.
\end{equation}
then $a_1=b_1$ and $c_1=d_1$.
\end{itemize}
Then there exists exactly one fixed point $(x^*,y^*)$ of system (\ref{f1f2}), and trajectory of any initial point $\left(x^{(0)}, y^{(0)}\right)$ in $[a,b]\times[c,d]$ converges to $(x^*,y^*)$.
\end{pro}
\begin{proof}
By the Brouwer fixed-point theorem, the operator
\[W:[a,b]\times[c,d]\rightarrow [a,b]\times[c,d]\]
defined by
\[W(x,y)=\left(f_1(x,y), f_2(x,y)\right)\]
has a fixed point $(x^*,y^*)$ (cf. with proof of Theorem 1.16 of \cite{GL}). Since $(x^*,y^*)$ also satisfies system (\ref{f1f2}), we can conclude that system (\ref{f1f2}) has at least one fixed point.

Consider the sequence  $\left(x^{(n)},y^{(n)}\right)$ generated by the operator $W$  for any initial point $\left(x^{(0)}, y^{(0)}\right)$ within the rectangle  $[a,b]\times[c,d]$. It suffices to show that
\[\lim\limits_{n\to \infty}\left(x^{(n)},y^{(n)}\right)=(x^*,y^*).\]
We denote $a_1^{(0)}=a, b_1^{(0)}=b, c_1^{(0)}=c, d_1^{(0)}=d$, and for each $i\geq 0$, let
\begin{equation}
\begin{array}{ll}
    a_1^{(i+1)}=f_1\left(a_1^{(i)},c_1^{(i)}\right), \ \ b_1^{(i+1)}=f_1\left(b_1^{(i)},d_1^{(i)}\right),\\[3mm]
    c_1^{(i+1)}=f_2\left(a_1^{(i)},c_1^{(i)}\right), \ \ d_1^{(i+1)}=f_2\left(b_1^{(i)},d_1^{(i)}\right).
\end{array}%
\end{equation}
Then
\[\begin{array}{ll}
    a_1^{(0)}=a\leq f_1\left(a_1^{(0)},c_1^{(0)}\right)\leq f_1\left(b_1^{(0)},d_1^{(0)}\right)\leq b=b_1^{(0)},\\[3mm]
    c_1^{(0)}=c\leq f_2\left(a_1^{(0)},c_1^{(0)}\right)\leq f_2\left(b_1^{(0)},d_1^{(0)}\right)\leq d=d_1^{(0)},
\end{array}\]
and from this, we can see that
\begin{center}
  $a_1^{(0)}\leq a_1^{(1)}\leq b_1^{(1)}\leq b_1^{(0)}$ \ and \  $c_1^{(0)}\leq c_1^{(1)}\leq d_1^{(1)}\leq d_1^{(0)}.$
\end{center}
Following the same approach, we have
\[\begin{array}{ll}
    a_1^{(1)}=f_1\left(a_1^{(0)},c_1^{(0)}\right)\leq f_1\left(a_1^{(1)},c_1^{(1)}\right)\leq f_1\left(b_1^{(1)},d_1^{(1)}\right)\leq f_1\left(b_1^{(0)},d_1^{(0)}\right)=b_1^{(1)},\\[3mm]
    c_1^{(1)}=f_2\left(a_1^{(0)},c_1^{(0)}\right)\leq f_2\left(a_1^{(1)},c_1^{(1)}\right)\leq f_2\left(b_1^{(1)},d_1^{(1)}\right)\leq f_2\left(b_1^{(0)},d_1^{(0)}\right)=d_1^{(1)},
\end{array}\]
and from this, we can see that
\begin{center}
  $a_1^{(0)}\leq a_1^{(1)}\leq a_1^{(2)}\leq b_1^{(2)}\leq b_1^{(1)}\leq b_1^{(0)}$ \ and \  $c_1^{(0)}\leq c_1^{(1)}\leq c_1^{(2)}\leq d_1^{(2)}\leq d_1^{(1)}\leq d_1^{(0)}.$
\end{center}
Note that
\begin{center}
  $a_1^{(0)}=a\leq x^{(n)}\leq b=b_1^{(0)}$ \ and \  $c_1^{(0)}=c\leq y^{(n)}\leq d=d_1^{(0)}$ for all $n\geq0.$
\end{center}
For all $n\geq0$, we have
\[\begin{array}{ll}
    a_1^{(1)}=f_1\left(a_1^{(0)},c_1^{(0)}\right)\leq f_1\left(x^{(n)},y^{(n)}\right)\leq f_1\left(b_1^{(0)},d_1^{(0)}\right)=b_1^{(1)},\\[3mm]
    c_1^{(1)}=f_2\left(a_1^{(0)},c_1^{(0)}\right)\leq f_2\left(x^{(n)},y^{(n)}\right)\leq f_2\left(b_1^{(0)},d_1^{(0)}\right)=d_1^{(1)}.
\end{array}\]
Hence, it follows that
\begin{center}
  $a_1^{(1)}\leq x^{(n)}\leq b_1^{(1)}$ \ and \  $c_1^{(1)}\leq y^{(n)}\leq d_1^{(1)}$ for all $n\geq1.$
\end{center}
For all $n\geq1$, we have
\[\begin{array}{ll}
    a_1^{(2)}=f_1\left(a_1^{(1)},c_1^{(1)}\right)\leq f_1\left(x^{(n)},y^{(n)}\right)\leq f_1\left(b_1^{(1)},d_1^{(1)}\right)=b_1^{(2)},\\[3mm]
    c_1^{(2)}=f_2\left(a_1^{(1)},c_1^{(1)}\right)\leq f_2\left(x^{(n)},y^{(n)}\right)\leq f_2\left(b_1^{(1)},d_1^{(1)}\right)=d_1^{(2)}.
\end{array}\]
This implies that
\begin{center}
  $a_1^{(2)}\leq x^{(n)}\leq b_1^{(2)}$ \ and \  $c_1^{(2)}\leq y^{(n)}\leq d_1^{(2)}$ for all $n\geq2.$
\end{center}
It follows by induction that for $i\geq0$, the following statements are true:
 \begin{flushleft}
 \textbf{(i)} \ \ \ $a=a_1^{(0)}\leq a_1^{(1)}\leq...\leq a_1^{(i-1)}\leq a_1^{(i)}\leq b_1^{(i)}\leq b_1^{(i-1)}\leq...\leq b_1^{(1)}\leq b_1^{(0)}=b.$\\[2mm]
\textbf{(ii)} \ \ $c=c_1^{(0)}\leq c_1^{(1)}\leq...\leq c_1^{(i-1)}\leq c_1^{(i)}\leq d_1^{(i)}\leq d_1^{(i-1)}\leq...\leq d_1^{(1)}\leq d_1^{(0)}=d.$\\[2mm]
\textbf{(iii)} \ $a_1^{(i)}\leq x^{(n)}\leq b_1^{(i)}$ \  for all $n\geq i.$  \\[2mm]
\textbf{(iv)} \ \ $c_1^{(i)}\leq y^{(n)}\leq d_1^{(i)}$ \ for all $n\geq i.$
\end{flushleft}
Assertions (i) and (ii) guarantee that the sequences $a_1^{(i)}$, $b_1^{(i)}$, $c_1^{(i)}$ and $d_1^{(i)}$ converge. Let their limits be denoted by $a_1$, $b_1$, $c_1$, and $d_1$, respectively.
Then
\begin{center}
  $a\leq a_1\leq b_1\leq b$ \ and \  $c\leq c_1\leq d_1\leq d.$
\end{center}
By the continuity of $f_1$ and $f_2$, we have

\[\left\{%
\begin{array}{ll}
    a_1=f_1(a_1,c_1), \\[2mm]
    c_1=f_2(a_1,c_1),
\end{array}%
\right.
\left\{%
\begin{array}{ll}
    b_1=f_1(b_1,d_1)\\[2mm]
    d_1=f_2(b_1,d_1)
\end{array}%
\right.\]
and hence $a_1=b_1$ and $c_1=d_1$
from which the proof follows.
\end{proof}

The following theorem describes the trajectory of any initial point $(x^{(0)}, y^{(0)})$ in $\Omega$.
\begin{thm}\label{Omega-0} Assume that (\ref{parametr}) holds and let $\beta<\nu$. Then the trajectory converges to the unique fixed point zero for any initial point $\left(x^{(0)}, y^{(0)}\right)\in\Omega$, i.e.,
\[\lim\limits_{n\to \infty}W_0^n\left(x^{(0)},y^{(0)}\right)=\left(0,0\right).\]
\end{thm}
\begin{proof} Let
\[f_1(x,y)=\frac{\beta y^2}{\gamma+y}+x\left(1-d_0-\frac{\alpha}{1+x}\right),\]
\[f_2(x,y)=\frac{\alpha x}{1+x}+(1-\mu)y.\]
Since functions $f_1(x, y)$ and $f_2(x, y)$ satisfy all the conditions of Proposition \ref{f1f2increasing}, the trajectory of any initial point within the set $\Omega$ converges to zero.
\end{proof}

\subsection{Dynamics on invariant sets.}\

A set $A$ is called invariant with respect to $W_{0}$ if $W_{0}(A)\subset A$. We denote
\[\Omega_1(x^*,y^*)=\left\{(x,y)\in\Omega, \ \ 0\leq x\leq x^*, \ \ 0\leq y\leq y^*\}\setminus\{(x^*,y^*)\right\},\]
\[\Omega_2(x^*,y^*)=\left\{(x,y)\in\Omega, \ \ x^*\leq x\leq\omega_1, y^*\leq y\leq\omega_2\right\}\setminus\left\{(x^*,y^*)\right\},\]
where $(x^*,y^*)$ is the positive fixed point defined by $(\ref{fix123})$.

\begin{lemma} The sets $\Omega_1$ and $\Omega_2$ are invariant with respect to $W_{0}.$
\end{lemma}
\begin{proof}

\textbf{1)} Let $0\leq x\leq x^*, 0\leq y\leq y^*.$ Then
\[x'-x^*=\frac{\beta y^2}{\gamma+y}+x(1-d_0-\frac{\alpha}{1+x})-x^*\leq\frac{\beta y^{*2}}{\gamma+y^*}+x^*(1-d_0-\frac{\alpha}{1+x})-x^*\]
\[=\frac{\beta y^{*2}}{\gamma+y^*}-d_0x^*-\frac{\alpha x^*}{1+x}=0,\]
\[y^*-y'=y^*-(1-\mu)y-\frac{\alpha x}{1+x}\geq y^*-(1-\mu)y^*-\frac{\alpha x^*}{1+x^*}=\mu y^*-\frac{\alpha x^*}{1+x^*}=0.\]
So, $(x',y')\in W_0(\Omega_1)\subset\Omega_1$

\textbf{2)} Let $x^*\leq x\leq \omega_1, y^*\leq y\leq\omega_2.$ Then
\[x^*-x'=x^*-\frac{\beta y^2}{\gamma+y}-x(1-d_0-\frac{\alpha}{1+x})\leq x^*-\frac{\beta {y^*}^2}{\gamma+y^*}-x^*(1-d_0-\frac{\alpha}{1+x})\]
\[=-\frac{\beta y^{*2}}{\gamma+y^*}+d_0x^*+\frac{\alpha x^*}{1+x}=0,\]
\[y'-y^*=\frac{\alpha x}{1+x}+(1-\mu)y-y^*\geq\frac{\alpha x^*}{1+x^*}+(1-\mu)y^*-y^*=0.\]
Therefore, $(x',y')\in W_0(\Omega_2)\subset\Omega_2.$
\end{proof}

\begin{lemma}\label{coop} The operator $W_0$ defined in (\ref{systemacase}) is monotone with respect to the $NE$ ordering, i.e., is a cooperative map.
\end{lemma}
\begin{proof} By (\ref{systemacase}) we have
\begin{equation}\label{cooper}
\begin{array}{ll}
x'_1-x'_2=\frac{\beta(y_1-y_2)\left(\gamma(y_1+y_2)+y_1y_2\right)}{(\gamma+y_1)(\gamma+y_2)}+(x_1-x_2)\left(1-d_0-\frac{\alpha}{(1+x_1)(1+x_2)}\right),\\
y'_1-y'_2=\frac{\alpha(x_1-x_2)}{(1+x_1)(1+x_2)}+(1-\mu)(y_1-y_2).
\end{array}
\end{equation}
Take $(x_1, y_1)\preceq_{NE}(x_2, y_2)\in\mathbb{R}^{2}_{+}$ then $x_1\leq x_2$ and $y_1\leq y_2$. By the last equalities from (\ref{cooper}) one can see that $x'_1\leq x'_2$ and $y'_1\leq y'_2$, i.e.,
\[(x'_1, y'_1)=W_0(x_1, y_1)\preceq_{NE}(x'_2, y'_2)=W_0(x_2, y_2).\]
This completes the proof.
\end{proof}

The following theorem  describes the trajectory of any point $(x^{(0)}, y^{(0)})$ in invariant sets $\Omega_1(x_1^*,y_1^*)$ and $\Omega_2(x_1^*,y_1^*)$.
\begin{thm}\label{2fix} Assume that (\ref{parametr}) holds and let $\beta=\nu$. For the operator $W_{0}$ and any initial point $(x^{(0)}, y^{(0)})$, the following hold
$$\lim_{n\to \infty}x^{(n)}=\left\{\begin{array}{ll}
0, \ \ \ \ \ \ \mbox{if} \ \ (x^{(0)}, y^{(0)})\in\Omega_1(x_1^*,y_1^*), \\[2mm]
x_1^*, \ \ \mbox{if} \ \ (x^{(0)}, y^{(0)})\in\Omega_2(x_1^*,y_1^*),
\end{array}\right.$$
$$\lim_{n\to \infty}y^{(n)}=\left\{\begin{array}{ll}
0, \ \ \ \ \ \ \mbox{if} \ \ (x^{(0)}, y^{(0)})\in\Omega_1(x_1^*,y_1^*), \\[2mm]
y_1^*, \ \ \ \ \  \mbox{if} \ \ (x^{(0)}, y^{(0)})\in\Omega_2(x_1^*,y_1^*),
\end{array}\right.$$
where $(x_1^*,y_1^*)$ is the fixed point defined by $(\ref{fix123})$.
\end{thm}
\begin{proof} For $\beta=\nu$ (see Theorem \ref{type}), $z_{1,1}=(0,0)$ is attracting, while $z_{1,2}=(x_1^*,y_1^*)$ is a non-hyperbolic.
By applying Proposition \ref{basin} to the ordered interval $\mathopen{[\![} z_{1,1}, z_{1,2} \mathclose{]\!]}_{NE}$, we can conclude that the trajectory of an arbitrary initial point in $\Omega_1(x_1^*,y_1^*)$ converges to zero. Proposition \ref{f1f2increasing} implies that the trajectory of any initial point selected from set $\Omega_2(x_1^*,y_1^*)$ converges to a limit point $(x^*_1, y^*_1)$. See Fig. 1(a) for numerical values and the visual illustration.
\end{proof}

The following theorem  describes the trajectory of any point $(x^{(0)}, y^{(0)})$ in invariant sets $\Omega_1(x_2^*,y_2^*)$ and $\Omega_2(x_2^*,y_2^*)$.
\begin{thm}\label{3fix} Assume that (\ref{parametr}) holds and let $\beta>\nu$. For the operator $W_{0}$ and any initial point $(x^{(0)}, y^{(0)})$, the following hold
$$\lim_{n\to \infty}x^{(n)}=\left\{\begin{array}{ll}
0, \ \ \ \ \ \ \mbox{if} \ \ (x^{(0)}, y^{(0)})\in\Omega_1(x_2^*,y_2^*), \\[2mm]
x_3^*, \ \ \mbox{if} \ \ (x^{(0)}, y^{(0)})\in\Omega_2(x_2^*,y_2^*),
\end{array}\right.$$
$$\lim_{n\to \infty}y^{(n)}=\left\{\begin{array}{ll}
0, \ \ \ \ \ \ \mbox{if} \ \ (x^{(0)}, y^{(0)})\in\Omega_1(x_2^*,y_2^*), \\[2mm]
y_3^*, \ \ \ \ \  \mbox{if} \ \ (x^{(0)}, y^{(0)})\in\Omega_2(x_2^*,y_2^*),
\end{array}\right.$$
where $(x_2^*,y_2^*)$, $(x_3^*,y_3^*)$ are the fixed point defined by $(\ref{fix123})$.
\end{thm}
\begin{proof}
Under the condition $\beta>\nu$ (see Theorem \ref{type}), the fixed points $z_{2,1}=(0,0)$ and $z_{2,3}=(x_3^*,y_3^*)$ are attracting, while $z_{2,2}=(x_2^*,y_2^*)$ is a repelling fixed point. Then, by
applying Proposition \ref{basin} to the ordered interval $\mathopen{[\![} z_{2,1}, z_{2,2} \mathclose{]\!]}_{NE}$, we can conclude that the trajectory of an arbitrary initial point in $\Omega_1(x_2^*,y_2^*)$ converges to zero.

Recall that $\omega_1$ and $\omega_2$ are defined by (\ref{Omega12}).
We divide the set $\Omega_2(x_2^*,y_2^*)$ into four parts:
\[\Omega^1(x_3^*,y_3^*)=\left\{(x,y)\in\Omega_2(x_2^*,y_2^*), \ \ x_3^*\leq x\leq \omega_1, \ \ y_3^*\leq y\leq \omega_2\right\},\]
\[\Omega^2(x_3^*,y_3^*)=\left\{(x,y)\in\Omega_2(x_2^*,y_2^*), \ \ x_2^*\leq x\leq x_3^*, \ \ y_3^*\leq y\leq \omega_2\right\},\]
\[\Omega^3(x_3^*,y_3^*)=\left\{(x,y)\in\Omega_2(x_2^*,y_2^*), \ \ x_2^*\leq x\leq x_3^*, \ \ y_2^*\leq x\leq y_3^*\right\},\]
\[\Omega^4(x_3^*,y_3^*)=\left\{(x,y)\in\Omega_2(x_2^*,y_2^*), \ \ x_3^*\leq x\leq \omega_1, \ \ y_2^*\leq x\leq y_3^*\right\}.\]
By applying Proposition \ref{basin} to the ordered interval $\mathopen{[\![} z_{2,2}, z_{2,3} \mathclose{]\!]}_{NE}$, we can conclude that the trajectory of an arbitrary initial point in $\Omega^3(x_3^*,y_3^*)$ converges to $(x^*_3, y^*_3)$. Proposition \ref{f1f2increasing} implies that the trajectory of any initial point in $\Omega^1(x_3^*,y_3^*)$ converges to $(x^*_3, y^*_3)$.

Assume that $\left(x^{(0)},y^{(0)}\right)\in int\left(\Omega^2(x_3^*,y_3^*)\right)$ or $\left(x^{(0)},y^{(0)}\right)\in int\left(\Omega^4(x_3^*,y_3^*)\right)$. Then, there exists $\left(\overline{x^{(0)}},\overline{y^{(0)}}\right)\in int\left(\Omega^3(x_3^*,y_3^*)\right)$ such that $$\left(\overline{x^{(0)}},\overline{y^{(0)}}\right)\preceq_{NE}\left(x^{(0)},y^{(0)}\right)$$ and $\left(\widetilde{x^{(0)}},\widetilde{y^{(0)}}\right)\in int\left(\Omega^1(x_3^*,y_3^*)\right)$ such that $$\left(x^{(0)},y^{(0)}\right)\preceq_{NE}\left(\widetilde{x^{(0)}},\widetilde{y^{(0)}}\right).$$ By monotonicity of $W_0$, we have
\[W_0^n\left(\overline{x^{(0)}},\overline{y^{(0)}}\right)\preceq_{NE}W_0^n\left(x^{(0)},y^{(0)}\right)\preceq_{NE}W_0^n\left(\widetilde{x^{(0)}},\widetilde{y^{(0)}}\right)\]
which implies $W_0^n\left(x^{(0)},y^{(0)}\right)\rightarrow (x^*_3, y^*_3)$ as $n\rightarrow\infty$. This completes the proof. See Fig. 1(a) for numerical values and the visual illustration.
\end{proof}

\begin{center}\label{fig1}
   \begin{tikzpicture} [scale=0.60]
   \draw[->, thick] (0,0) -- (0,8);
   \draw[->, thick] (0,0) -- (11,0);
   \draw[-, thick] (3,0) -- (3,7.05);
   \draw[-, thick] (0,3) -- (10,3);
   \draw[thick] (0,7) --  (10, 7);
   \draw[thick] (10,0) --  (10, 7);

   \fill[brown, opacity=0.3] (0,0) -- (3,0) -- (3,3) -- (0,3) -- cycle;
   \fill[yellow!50, opacity=0.3] (3,3) -- (10,3) -- (10,7) -- (3,7) -- cycle;
   \filldraw[red] (0,0) circle (3pt);
   \filldraw[red] (3,3) circle (3pt);
\filldraw[black] (2.8,2.9) circle (1.5pt);
\filldraw[black] (0.2,2.9) circle (1.5pt);
\filldraw[black] (2.7, 0.5) circle (1.5pt);
\filldraw[black] (9.7, 4) circle (1.5pt);
\filldraw[black] (9, 6.8) circle (1.5pt);
\filldraw[black] (4, 6.8) circle (1.5pt);
\filldraw[black] (9.8, 3.1) circle (1.5pt);
   \node[left] at (10.5,-0.5){$\omega_1$};
   \node[above] at (11,0){$x$};
   \node[left] at (3.3,-0.5){$x^*_1$};

   \node[left] at (0,7){$\omega_2$};
   \node[right] at (0,8){$y$};
   \node[left] at (0,3){$y^*_1$};
   \node[left] at (0,-0.5){$0$};
\node[above] at (5,-2){$(a)$};
   \draw[->, thick, dotted, line width=1pt] (2.8,2.9).. controls (2,2.7) and (0.8,2.6) .. (0.2, 0.2);
   \draw[->, thick, dotted, line width=1pt] (0.2,2.9).. controls (0.9,2.8)  .. (0.8, 2.1);
   \draw[->, thick, dotted, line width=1pt] (2.7, 0.5).. controls (2.5,1.5) and (2,2) .. (0.8, 1.2);

  \draw[->, thick, dotted, line width=1pt] (9.7,4).. controls (9,4) and (7,3.9).. (3.3, 3.2);
  \draw[->, thick, dotted, line width=1pt] (9,6.8).. controls (9,6) and (8.8,5) .. (6.9, 4);
  \draw[->, thick, dotted, line width=1pt] (4,6.8).. controls (6.2,6.3) and (6.2, 5) .. (4.5, 3.5);
  \draw[->, thick, dotted, line width=1pt] (9.8,3.1).. controls (8.9,3.7) .. (7.9, 3.6);
     \end{tikzpicture}
\begin {tikzpicture} [scale=0.60]
   \draw[->, thick] (0,0) -- (0,8);
   \draw[->, thick] (0,0) -- (11,0);
\draw[-, thick] (2.5,0) -- (2.5,7.05);
   \draw[-, thick] (0,2) -- (10,2);

   \draw[-, thick, dashed] (5.5,0) -- (5.5,7.05);
\draw[-, thick, dashed] (0,3.5) -- (10,3.5);

   \draw[thick] (0,7) --  (10, 7);
   \draw[thick] (10,0) --  (10, 7);
   \fill[brown, opacity=0.3] (0,0) -- (2.5,0) -- (2.5,2) -- (0,2) -- cycle;
   \fill[yellow!50, opacity=0.3] (2.5,2) -- (10,2) -- (10,7) -- (2.5,7) -- cycle;

\filldraw[red] (2.5,2) circle (3pt);
\filldraw[red] (0,0) circle (3pt);
\filldraw[red] (5.5,3.5) circle (3pt);
\filldraw[black] (2.1,1.8) circle (1.5pt);
\filldraw[black] (0.2,1.9) circle (1.5pt);
\filldraw[black] (2.4, 0.5) circle (1.5pt);
\filldraw[black] (2.8,2.2) circle (1.5pt);
\filldraw[black] (2.8, 5.5) circle (1.5pt);
\filldraw[black] (6.2,2.2) circle (1.5pt);
\filldraw[black] (9.7,4.2) circle (1.5pt);
\filldraw[black] (5,6.7) circle (1.5pt);
\filldraw[black] (9.5,2.2) circle (1.5pt);
   \node[left] at (10.5,-0.5){$\omega_1$};
   \node[above] at (11,0){$x$};
   \node[above] at (2.5,-1){$x^*_2$};
   \node[above] at (5.5,-1){$x^*_3$};
   \node[left] at (0,7){$\omega_2$};
   \node[right] at (0,8){$y$};
    \node[left] at (0,-0.5){$0$};
\node[left] at (0,2){$y^*_2$};
   \node[left] at (0,3.5){$y^*_3$};

   \node[above] at (5,-2){$(b)$};
    \draw[->, thick, dotted, line width=1pt] (2.1,1.8).. controls (2.2,1.8) and (1.2,1.6) .. (0.3, 0.25);
  \draw[->, thick, dotted, line width=1pt] (0.2,1.9).. controls (0.9,1.8)  .. (0.8, 1.1);
   \draw[->, thick, dotted, line width=1pt] (2.4, 0.5).. controls (2,1.6) and (1.2,1.3) .. (1, 0.9);

    \draw[->, thick, dotted, line width=1pt] (2.8,2.2) .. controls (3,2.3) and (4,3) ..(5.3, 3.4);
    \draw[->, thick, dotted, line width=1pt] (2.8, 5.5).. controls (3,4) and (4,3.5) ..(4.8, 3.3);
    \draw[->, thick, dotted, line width=1pt] (6.2,2.2).. controls (5.3,2.3) and (5, 2.5) .. (5.1, 3.1);

  \draw[->, thick, dotted, line width=1pt] (9.7,4.2) .. controls (8,4.1) and (6,3.7) ..(5.7, 3.6);
  \draw[->, thick, dotted, line width=1pt] (5,6.7).. controls (6.7,6.5) and (7,5) .. (5.9, 3.8);
  \draw[->, thick, dotted, line width=1pt] (9.5,2.2).. controls (9,4) and (8,4) .. (7, 3.7);
\end{tikzpicture}
\end{center}
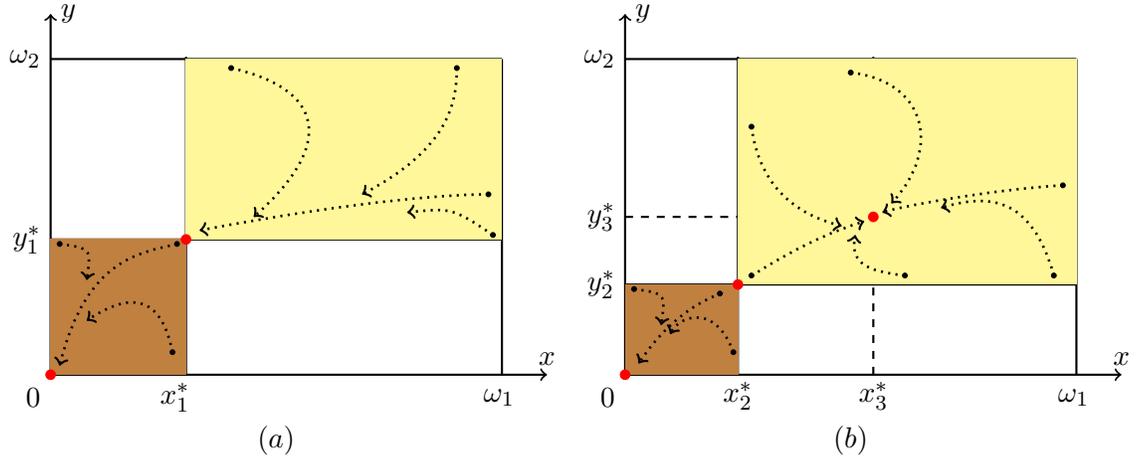
{\footnotesize{{Figure 1.}(a) Visual illustration of Theorem \ref{2fix} when $\beta=\nu$, $\alpha=0.4, \beta=\frac{9}{4}(2+\sqrt{3}), \gamma=1, d_0=0.5, \mu=0.6$, in the case when $(x_1^*,y_1^*)$ is a non-hyperbolic point. (b) Visual illustration of Theorem \ref{3fix} when $\beta>\nu$, $\alpha=0.4, \beta=\frac{9}{4}(2+\sqrt{3}), \gamma=1, d_0=0.5, \mu=0.6$. The case when $(x_2^*,y_2^*)$ is a repeller and $(0,0)$ and $(x_3^*,y_3^*)$ are attracting points.}}

\subsection{On the set $\Omega\setminus\left(\Omega_1(x^*,y^*)\bigcup\Omega_2(x^*,y^*)\right)$}\

In the following examples, we show trajectories of initial points from the set $\Omega\setminus\left(\Omega_1\bigcup\Omega_2\right)$.

Let us consider the operator $W_0$ with parameter values $\alpha=0.4, \beta=\frac{9}{4}(2+\sqrt{3}), \gamma=1, d_0=0.5, \mu=0.6,$ satisfying the condition $\beta=\nu$.
 Then by (\ref{fix123}), we get $\left(x_1^*,y_1^*\right)=\left(\frac{3\sqrt{3}}{5}, 9-5\sqrt{3}\right)$.
\begin{ex}\label{exam1}
For initial points $z_1^{(0)}=(0.1,0.5)$ and $z_2^{(0)}=(2.5,0.1)$, the trajectory of the system (\ref{systemacase}) is shown in Fig.2(a). In this case, the system converges to the origin, i.e., $$\lim_{n\to \infty}x^{(n)}=0, \ \lim_{n\to \infty}y^{(n)}=0.$$
For initial points $z_3^{(0)}=(0.1,0.6)$ and $z_4^{(0)}=(3,0.1)$, the trajectory of the system (\ref{systemacase}) is also shown in Fig.2(a) (refer to the same figure for both cases). Here, the first coordinate of the trajectory converges to $\frac{3\sqrt{3}}{5}$, while the second coordinate has a limit point of $9-5\sqrt{3}$:
$$\lim_{n\to \infty}x^{(n)}=\frac{3\sqrt{3}}{5}, \ \lim_{n\to \infty}y^{(n)}=\frac{\alpha}{\mu}=9-5\sqrt{3}.$$
\end{ex}

We consider the operator $W_0$ with the following parameter values: $\alpha=0.4, \beta=9, \gamma=1, d_0=0.5,$ and $\mu=0.6$. Notably, these parameters satisfy the condition $\beta>\nu$.  By (\ref{fix123}), we obtain the fixed points $\left(x_2^*,y_2^*\right)=\left(\frac{3}{5}, \frac{1}{4}\right)$ and $\left(x_3^*,y_3^*\right)=\left(\frac{9}{5}, \frac{3}{7}\right)$.

\begin{ex}\label{exam2}
Figure 2(b) illustrates the trajectories of system (\ref{systemacase}) for initial points $z_1^{(0)}=(0.1,0.3)$ and $z_2^{(0)}=(1,0.1)$. In both cases, the system converges to the origin.

Considering initial points $z_3^{(0)}=(0.5,0.5)$,  $z_4^{(0)}=(0.5,0.66)$, $z_5^{(0)}=(2,0.1)$, and $z_6^{(0)}=(4,0.3)$, the trajectory of system (\ref{systemacase}) is shown in Fig. 2(b). In this case, the first coordinate of the trajectory appears to converge towards $\frac{9}{5}$, while the second coordinate appears to have a limit point of  $\frac{3}{7}$.
\end{ex}

\begin{center}\label{fig2}
   \begin{tikzpicture} [scale=0.60]
   \draw[->, thick] (0,0) -- (0,8);
   \draw[->, thick] (0,0) -- (11,0);
   \draw[-, thick] (3,0) -- (3,7.05);
   \draw[-, thick] (0,3) -- (10,3);
   \draw[thick] (0,7) --  (10, 7);
   \draw[thick] (10,0) --  (10, 7);

   \fill[brown, opacity=0.3] (0,0) -- (3,0) -- (3,3) -- (0,3) -- cycle;
   \fill[yellow!50, opacity=0.3] (3,3) -- (10,3) -- (10,7) -- (3,7) -- cycle;
   \filldraw[red] (0,0) circle (3pt);
   \filldraw[red] (3,3) circle (3pt);
\filldraw[black] (0.2,5) circle (1.5pt);
\filldraw[black] (8, 0.5) circle (1.5pt);
\filldraw[black] (4.2, 0.5) circle (1.5pt);
\filldraw[black] (2.6, 2.8) circle (1.5pt);
\filldraw[black] (2, 6.5) circle (1.5pt);
   \node[left] at (10.5,-0.5){$\omega_1$};
   \node[above] at (11,0){$x$};
   \node[left] at (3.3,-0.5){$x^*_1$};

   \node[left] at (0,7){$\omega_2$};
   \node[right] at (0,8){$y$};
   \node[left] at (0,3){$y^*_1$};
   \node[left] at (0,-0.5){$0$};
\node[above] at (5,-2){$(a)$};
\node[right] at (-0.1,5.4){$z_1^{(0)}$};
\node[right] at (4.1,0.6){$z_2^{(0)}$};
\node[right] at (1.8,6.92){$z_3^{(0)}$};
\node[right] at (8.0,0.7){$z_4^{(0)}$};
   \draw[->, thick, dotted, line width=1pt] (2.6,2.8).. controls (2,2.7) and (0.8,2.6) .. (0.2, 0.2);
   \draw[->, thick, dotted, line width=1pt] (0.2,5).. controls (1.4,4.4)  .. (2.5, 2.8);
   \draw[->, thick, dotted, line width=1pt] (4.2, 0.5).. controls (3.7,2.7) and (2.9,2.5) .. (2.6, 2.7);

   \draw[->, thick, dotted, line width=1pt] (6,3.5)-- (3.3, 3.1);
  \draw[->, thick, dotted, line width=1pt] (2,6.5).. controls (3 ,6.3) and (4, 5) .. (3.5, 3.4);
  \draw[-, thick, dotted, line width=1pt] (8,0.5).. controls (7,3.5) .. (6, 3.5);
     \end{tikzpicture}
\begin {tikzpicture} [scale=0.60]
   \draw[->, thick] (0,0) -- (0,8);
   \draw[->, thick] (0,0) -- (11,0);
\draw[-, thick] (2.5,0) -- (2.5,7.05);
   \draw[-, thick] (0,2) -- (10,2);

   \draw[-, thick, dashed] (5.5,0) -- (5.5,7.05);
\draw[-, thick, dashed] (0,3.5) -- (10,3.5);

   \draw[thick] (0,7) --  (10, 7);
   \draw[thick] (10,0) --  (10, 7);
   \fill[brown, opacity=0.3] (0,0) -- (2.5,0) -- (2.5,2) -- (0,2) -- cycle;
   \fill[yellow!50, opacity=0.3] (2.5,2) -- (10,2) -- (10,7) -- (2.5,7) -- cycle;

\filldraw[red] (2.5,2) circle (3pt);
\filldraw[red] (0,0) circle (3pt);
\filldraw[red] (5.5,3.5) circle (3pt);
\filldraw[black] (2.1,1.8) circle (1.5pt);
\filldraw[black] (0.2,2.9) circle (1.5pt);
\filldraw[black] (3.4, 0.3) circle (1.5pt);
\filldraw[black] (2, 5.2) circle (1.5pt);
\filldraw[black] (6.2,0.3) circle (1.5pt);
\filldraw[black] (2,6.7) circle (1.5pt);
\filldraw[black] (9.5,0.6) circle (1.5pt);
   \node[left] at (10.5,-0.5){$\omega_1$};
   \node[above] at (11,0){$x$};
   \node[above] at (2.5,-1){$x^*_2$};
   \node[above] at (5.5,-1){$x^*_3$};
   \node[left] at (0,7){$\omega_2$};
   \node[right] at (0,8){$y$};
    \node[left] at (0,-0.5){$0$};
\node[left] at (0,2){$y^*_2$};
   \node[left] at (0,3.5){$y^*_3$};
   \node[right] at (-0.1,3.35){$z_1^{(0)}$};
\node[right] at (3.25,0.55){$z_2^{(0)}$};
\node[right] at (1.15,5.4){$z_3^{(0)}$};
\node[right] at (1.1,6.94){$z_4^{(0)}$};
\node[right] at (6.15,0.5){$z_5^{(0)}$};
\node[right] at (9.44,0.75){$z_6^{(0)}$};

   \node[above] at (5,-2){$(b)$};
    \draw[->, thick, dotted, line width=1pt] (2.1,1.8).. controls (2.2,1.8) and (1.2,1.6) .. (0.3, 0.25);
  \draw[->, thick, dotted, line width=1pt] (0.2,2.9).. controls (0.9,2.8)  .. (2, 1.9);
   \draw[->, thick, dotted, line width=1pt] (3.4, 0.3).. controls (3,1) and (2.5,1.6) .. (2.2, 1.7);

    \draw[->, thick, dotted, line width=1pt] (4.8,3.3) --(5.3, 3.4);
    \draw[-, thick, dotted, line width=1pt] (2, 5.2).. controls (3,4) and (4,3.5) ..(4.8, 3.3);
    \draw[->, thick, dotted, line width=1pt] (6.2,0.3).. controls (5.5,0.9) and (3, 2.5) .. (5.1, 3.1);

  \draw[->, thick, dotted, line width=1pt] (7,4) -- (6.2, 3.7);
  \draw[->, thick, dotted, line width=1pt] (2,6.7).. controls (6.7,6.5) and (7,5) .. (5.9, 3.8);
  \draw[-, thick, dotted, line width=1pt] (9.5,0.6).. controls (9,4) and (8,4) .. (7, 4);
\end{tikzpicture}
\end{center}
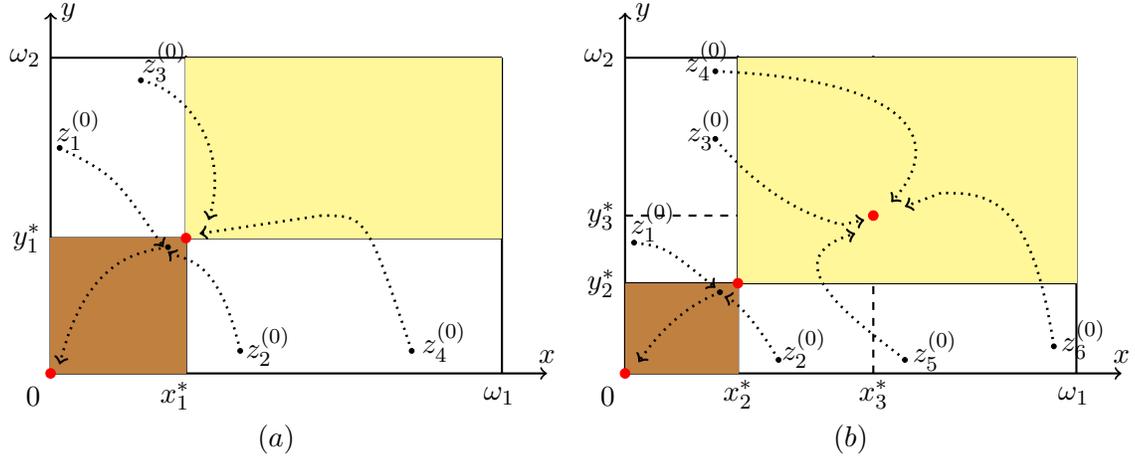
{\footnotesize{{Figure 2.}(a) Visual illustration of Example \ref{exam1} when $\beta=\nu$, and initial points $z_1^{(0)}=(0.1,0.5)$, $z_2^{(0)}=(2.5,0.1)$, $z_3^{(0)}=(0.1,0.6)$, $z_4^{(0)}=(3,0.1)$. (b) Visual illustration of Example \ref{exam2} when $\beta>\nu$, and initial points $z_1^{(0)}=(0.1,0.3)$, $z_2^{(0)}=(1,0.1)$, $z_3^{(0)}=(0.5,0.5)$, $z_4^{(0)}=(0.5,0.66)$, $z_5^{(0)}=(2,0.1)$, $z_6^{(0)}=(4,0.3)$.}}

\section{Discussion and conclusion}
We note that for continuous time system (\ref{system}) the following results are known (see \cite{J.Li}):  All fixed (equilibrium) points are found and their types are determined. Moreover, local behavior of the dynamical system in the neighborhood of the fixed point $(0,0)$ is studied. In our discrete-time case, for $d_0>0$ we also determined types of all fixed points. Besides this we have been able to study global behavior of the system in neighborhood of $(0,0)$. Also, we have found invariant sets and studied the dynamical system on the sets. The last results are not known for the continuous time.

%
%

\end{document}